\documentclass{icmart}
\usepackage{cite}
\usepackage{breqn}
\usepackage{tikz}
\usetikzlibrary{arrows,calc}
\tikzset{
>=stealth',
help lines/.style={dashed, thick}, axis/.style={<->}, important
line/.style={thick}, connection/.style={thick, dotted}, }

\contact[ostrover@post.tau.ac.il]{School of Mathematical Sciences, 
Tel Aviv University, Ramat Aviv 69978 Israel}

\newtheorem{theorem}{Theorem}[section]
\newtheorem{corollary}[theorem]{Corollary}
\newtheorem{lemma}[theorem]{Lemma}

\newtheorem{conjecture}[theorem]{Conjecture}
\newtheorem{question}[theorem]{Question}

\theoremstyle{definition}
\newtheorem{definition}[theorem]{Definition}
\newtheorem{remark}[theorem]{Remark}

\title[When Symplectic Topology Meets  Banach Space Geometry]{When  Symplectic Topology Meets  Banach Space Geometry}

\author[Yaron Ostrover]
{Yaron Ostrover
}

\begin{document}

\begin{abstract}
In this paper we survey some recent 
works that take the first steps toward establishing  
bilateral connections between symplectic geometry and 
several other fields, namely,  asymptotic geometric analysis, classical convex geometry, and the theory of normed spaces. 

\end{abstract}



\maketitle

\section{Introduction}

 In the last three decades, 
symplectic topology has had an astonishing amount of  fruitful interactions with other fields of mathematics, including complex and algebraic geometry,  dynamical systems, 
Hamiltonian PDEs, transformation groups, 
and low-dimensional topology; 
as well as with physics, where, for example, symplectic topology plays a key role in the  rigorous formulation of mirror symmetry. 

\smallskip

In this survey paper, we present some recent works that take first steps toward establishing novel interrelations between symplectic geometry and 
several fields of mathematics, 
namely, asymptotic geometric analysis, classical convex geometry,  and the theory of normed spaces. 
In the first part of this paper (Sections~\ref{SEC:ISOPERIMTERIC} and~\ref{SEC:MAHLER}) we concentrate on the theory of symplectic measurements, 
which arose from the foundational work of Gromov~\cite{Gr} on pseudoholomorphic curves;  followed by the seminal works of Ekeland and Hofer~\cite{EH} and Hofer and Zehnder~\cite{HZ1} on variational theory 
in Hamiltonian systems, and Viterbo on generating  functions~\cite{V1}. This theory -- also known as the theory of ``symplectic capacities" -- lies nowadays at the core of symplectic geometry and topology.

\smallskip

In Section~\ref{SEC:ISOPERIMTERIC}, we focus on an open symplectic isoperimetric-type conjecture proposed  by Viterbo in~\cite{V}. It states that among all convex domains 
 with a given volume in the classical phase space ${\mathbb R}^{2n}$, the Euclidean ball has the maximal ``symplectic size" (see Section~\ref{SEC:ISOPERIMTERIC} below for the precise statement). 
In a collaboration with S. Artstein-Avidan and V. D.  Milman~\cite{AMO},   
we were able to prove an asymptotic version of Viterbo's conjecture, that is, we proved the conjecture up to a universal (dimension-independent) constant. This has been achieved by  adapting 
techniques from asymptotic geometric analysis and adjusting them
to a symplectic context, while working exclusively in the linear symplectic category.

\smallskip

The fact that one can get within a constant factor to the full conjecture using only linear embeddings is somewhat surprising from the symplectic-geometric point of view, as in symplectic geometry one typically 
needs highly nonlinear tools to estimate capacities.
However, this  fits perfectly into the philosophy of asymptotic geometric analysis.
Finding dimension independent estimates is a
frequent goal in this field, 
where surprising phenomena such
as concentration of measure (see e.g.~\cite{MilSch}) imply the existence of order and structures
in high dimensions, despite the huge complexity it involves. 
It would be interesting to explore whether similar phenomena also exist in the framework of symplectic geometry.
A natural important source for the study of the asymptotic  behavior (in the dimension) of symplectic invariants is the field of statistical mechanics,  
where one considers systems with a large number of particles, and the dimension of the phase space is twice the number of degrees of freedom. 
It seems that symplectic measurements were overlooked in this context so far.

In Section~\ref{SEC:MAHLER} we go in the opposite direction: we show how symplectic geometry could potentially be used to tackle  a 70-years-old fascinating open question
in convex geometry, 
known as the Mahler conjecture.
Roughly speaking, Mahler's conjecture states that the minimum of the product of the volume of a centrally symmetric convex body and the volume of its polar body is attained (not uniquely) for the hypercube. 
In a collaboration with S. Artstein--Avidan and R. Karasev~\cite{AKO}, we combined tools from symplectic geometry, classical convex analysis, and the theory of mathematical billiards,
and established a close relation between Mahler's conjecture and the above mentioned symplectic isoperimetric conjecture by Viterbo. More preciesly, we showed that Mahler's conjecture is equivalent to a special case of Viterbo's conjecture (see Section~\ref{SEC:MAHLER} for details). 

In the second part of the paper (Section~\ref{SEC:HOFER}),  we explain how methods from functional analysis can be used to address  questions regarding the 
 geometry of the group ${\rm Ham}(M, \omega)$ of Hamiltonian
diffeomorphisms   associated with a symplectic manifold $(M,\omega)$. 
One of the most striking facts regarding
this group, discovered by Hofer in~\cite{H}, is that it carries an intrinsic geometry given by a Finsler bi-invariant metric, nowadays 
known as Hofer's metric. This metric measures the time-averaged minimal oscillation of a
Hamiltonian function that is needed to generate a Hamiltonian diffeomorphism starting from the identity. 
Hofer's metric has been
intensively studied in the past twenty years, leading to many discoveries covering a wide
range of subjects from Hamiltonian dynamics to symplectic topology (see e.g.,~\cite{HZ, Mcd1,P1} and the references therein).
A long-standing question raised by Eliashberg and Polterovich in~\cite{EliP} is whether Hofer's metric is the only bi-invariant Finsler metric on the group ${\rm Ham}(M, \omega)$. 
Together with L. Buhovsky~\cite{BO}, and based on previous results by Ostrover and Wagner~\cite{OW}, we used methods from functional analysis and the theory of normed function spaces to  affirmatively answer this question. 
We proved that any non-degenerate bi-invariant Finsler metric on ${\rm Ham}(M, \omega)$, which is
generated by a norm that is continuous in the $C^{\infty}$-topology, gives rise to the same
topology on ${\rm Ham}(M, \omega)$ as the one induced by Hofer's metric.


\smallskip
As mentioned before, the outlined interdisciplinary connections described above  are just the first few steps in what seems to be a promising new direction. We hope that further exploration of these connections will strengthen the dialogue between these fields and symplectic geometry, and expand the range of methodologies alongside research questions that can be tackled through these means.
%
%
%

%
%

\smallskip

We end this paper with several open questions and speculations regarding some of the mentioned topics (see Section~\ref{SEC:OQ}).

\section{A Symplectic Isoperimetric Inequality} \label{SEC:ISOPERIMTERIC}

A classical result in symplectic geometry (Darboux's theorem) states that symplectic manifolds - in a sharp contrast
to Riemannian manifolds - have no local invariants (except, of course, the dimension). The first examples
of global symplectic invariants were introduced by Gromov in his seminal paper~\cite{Gr}, where
he developed and used pseudoholomorphic curve techniques to prove a striking symplectic rigidity result.
Nowadays known as Gromov's ``non-squeezing theorem", this result states that one cannot map a ball inside a thinner cylinder
by a symplectic embedding. 
This theorem paved the way to the introduction
of global symplectic invariants, called symplectic capacities which, roughly speaking, measure the
symplectic size of a set.

\smallskip

We will focus here on the case of the classical phase space ${\mathbb R}^{2n} \simeq {\mathbb C}^n$  equipped with the standard symplectic structure $\omega=dq \wedge dp$. 
We denote by $B^{2n}(r)$ the Euclidean ball of radius $r$, and by $Z^{2n}(r)$ the cylinder $B^{2}(r) \times {\mathbb C}^{n-1}$. 
Gromov's non-squeezing theorem asserts that if $r < 1$ there is no symplectomorphism $\psi$ of ${\mathbb R}^{2n}$ such that $\psi(B^{2n}(1)) \subset Z^{2n}(r)$.
The following definition, which crystallizes the notion of  ``symplectic size", was given by Ekeland and Hofer in their influential paper~\cite{EH}.

%

\smallskip

\noindent {\bf Definition:} 
A symplectic capacity on $({\mathbb R}^{2n},\omega)$ associates
to each  subset $U \subset {\mathbb R}^{2n}$ a number $c(U) \in
[0,\infty]$ such that the following three properties hold:

\smallskip

\noindent (P1) $c(U) \leq c(V)$ for $U \subseteq V$ (monotonicity);

\smallskip

\noindent (P2) $c \big (\psi(U) \big )= |\alpha| \, c(U)$ for  $\psi
\in {\rm Diff} ( {\mathbb R}^{2n} )$ such that $\psi^*\omega=
\alpha  \omega$ (conformality);

\smallskip

\noindent (P3) $c \big (B^{2n}(r) \big ) = c \big (Z^{2n}(r) 
 \big ) = \pi r^2$ (nontriviality and
normalization).

\smallskip

Note that (P3)  disqualifies any volume-related
invariant, while (P1) and (P2) imply that for $U, V \subset {\mathbb
R}^{2n}$, a necessary condition for the existence of a
symplectomorphism $\psi $ 
 with $\psi(U) = V$, is $c(U) =c(V)$ for any symplectic capacity $c$.

%

\smallskip

It is a priori unclear that symplectic capacities exist.
The above mentioned non-squeezing result  naturally leads to the definition of two symplectic capacities: 
the Gromov radius, defined by $\underline c(U)=\sup\{\pi r^2 \, | \, B^{2n}(r) \stackrel{\rm s} \hookrightarrow U \} $; and the 
cylindrical capacity, defined by $\overline c(U) = \inf\{\pi r^2  \, | \, U \stackrel{\rm s} \hookrightarrow Z^{2n}(r) \} $, where $\stackrel{\rm s} \hookrightarrow$ stands for symplectic embedding. 
It is easy to verify that these two capacities are the smallest and largest  possible symplectic capacities, respectively. 
Moreover, it is also known that the existence of a single capacity readily implies Gromov's non-squeezing theorem, as well as 
the Eliashberg-Gromov $C^0$-rigidity theorem, which states that for any closed symplectic manifold $(M,\omega)$,  the  symplectomorphism group ${\rm Symp}(M,\omega)$  is $C^0$-closed in the group  of all  diffeomorphisms of $M$  (see e.g.,  Chapter 2 of~\cite{HZ}). 

\smallskip

Shortly after Gromov's work, other symplectic capacities were constructed,
such as the Hofer-Zehnder~\cite{HZ} and the Ekeland-Hofer~\cite{EH} capacities, the displacement energy~\cite{H},  the Floer-Hofer capacity~\cite{FH,FHW},
spectral capacities~\cite{FGS,Oh,V1}, and, more recently, Hutchings's embedded contact homology capacities~\cite{Hu1}.
Nowadays, symplectic capacities are among the most fundamental objects in symplectic geometry,
and are the subject of intensive research efforts (see e.g.,~\cite{Hu2,IrirKei1, L, LMT, LMS, Lu, Mcd2,McSch,Schle}, and~\cite{CHLS} for a recent detailed survey and more references).
 However, in spite of the rapidly accumulating
knowledge regarding symplectic capacities, they are 
notoriously difficult
 to compute, and
there are no general methods even to effectively estimate them.
%

\smallskip


In~\cite{V}, Viterbo investigated the relation between the
symplectic way of measuring the size of sets using symplectic
capacities, and the classical  approach using  volume. Among many other inspiring results, in that work he
conjectured that in the class of convex bodies in ${\mathbb R}^{2n}$
with fixed volume, the Euclidean ball $B^{2n}$ maximizes any
given symplectic capacity. More precisely,

\begin{conjecture}[Viterbo's volume-capacity inequality conjecture] \label{iso-per-conj} For
any convex body $K $ in ${\mathbb R}^{2n}$ and any symplectic
capacity $c$,
\begin{equation*} \label{AAMO-result} {\frac {c(K)} {c(B)}} \leq  \left  (   {\frac {{\rm Vol}(K)}
{{\rm Vol}(B)}} \right )^{1/n}, \ \ {\rm where} \ B = B^{2n}(1).
\end{equation*}
\end{conjecture}

Here and henceforth a convex body of ${\mathbb R}^{2n}$ is a compact convex set with non-empty interior. 
The isoperimetric inequality above was proved in~\cite{V} up to a
constant that depends linearly on the dimension using the classical John ellipsoid theorem. In a joint work
with S. Artstein-Avidan and V. D. Milman (see~\cite{AMO}), we made further progress towards the proof
of the conjecture. By customizing 
methods and
techniques from asymptotic geometric analysis and adjusting them to
the symplectic context, we were able to prove Viterbo's conjecture up
to a universal  (i.e., dimension-independent) constant. More
precisely, we proved that 
\begin{theorem} \label{up-to-uni-constnat} There is a universal constant $A$ such that for
any convex domain $K$ in ${\mathbb R}^{2n}$, and any symplectic
capacity $c$, one has
$$ {\frac {c(K)} {c(B)}} \leq A \, \left (   {\frac {{\rm Vol}(K)} {{\rm Vol}(B)}} \right )^{1/n}, \ \ {\rm where} \ B = B^{2n}(1).$$
\end{theorem}

We emphasize that 
in the proof of Theorem~\ref{up-to-uni-constnat}  we work exclusively in the category of linear symplectic geometry.
It turns out that even in this limited category of linear symplectic transformations, there are tools which are powerful enough to obtain a dimension-independent estimate as above.
While this fits with the philosophy of asymptotic geometric analysis, it is less expected from a
symplectic geometry point of view, where  one expects that 
highly nonlinear methods, such as folding and wrapping techniques (see e.g., the book~\cite{Schle}), would be required to effectively estimate symplectic capacities.

\smallskip

The proof of Theorem~\ref{up-to-uni-constnat} above is based on two ingredients. The first is the following simple geometric observation (see Lemma 3.3 in~\cite{AMO}, cf.~\cite{APB}).
\begin{lemma} \label{lem-complex-symetric} 
If a convex body  $K \subset {\mathbb C}^{n}$  satisfies $K=iK$, then  $\overline c(K) \leq {\frac 4 {\pi}} \, \underline c(K)$.
\end{lemma}
\begin{proof}[Sketch of Proof]
Let $rB^{2n}$ be the  largest multiple  of the unit ball contained in $K$, and let $x  \in \partial K \cap rS^{2n-1}$ be a contact point between the boundary of $K$ and the boundary of  $rB^{2n}$.
It follows from the convexity assumption that the body $K$ lies between the hyperplanes $x + x^{\perp}$
and $-x + x^{\perp}$. Moreover,  since $K=iK$, it lies also between $-ix + ix^{\perp}$ and $ix +
ix^{\perp}$. Thus, the projection of $K$ onto the plane spanned
by $x$ and $ix$ is contained in a square of edge length $2r$. This square  can be turned into a disc with area $4r^2$, after applying a non-linear
symplectomorphism which is essentially two-dimensional.  Therefore, $K$
is contained in a symplectic image of the cylinder $Z^{2n}(\sqrt{4/{\pi}}\,r)$, and the lemma follows.
\end{proof}
Since by monotonicity, Conjecture~\ref{iso-per-conj} trivially holds for the Gromov radius $ \underline c$, it follows from Lemma~\ref{lem-complex-symetric}  that
%
%
\begin{corollary} \label{COR:about-sym-bodies}
Theorem~\ref{up-to-uni-constnat} holds for 
convex bodies $K \subset {\mathbb C}^{n}$ such that $K=iK$.  

\end{corollary}
\smallskip

The second ingredient in the proof is  a profound result in asymptotic geometric analysis  discovered by V.D. Milman in the mid 1980's called the ``reverse Brunn-Minkowski inequality" (see~\cite{Mil,Milm1317}).
Recall that the classical Brunn-Minkowski
inequality states that if $A$ and $B$ are non-empty Borel
subsets of ${\mathbb R}^n$, then
$$ {\rm Vol} (A+B)^{1/n} \geq {\rm Vol}(A)^{1/n} + {\rm
Vol}(B)^{1/n},$$ 
where $A+B = \{x+y \, | \, x \in A, \, y \in B \}$ is the Minkowski sum. 
Although at  first glance it seems that one cannot
expect any inequality in the reverse direction (consider, e.g.,  two very long and thin ellipsoids pointing in orthogonal
directions in ${\mathbb R}^2$), it turns out that for convex bodies, if one allows for an extra choice of
``position'', i.e., a volume-preserving linear image of the bodies, then one can reverse the Brunn-Minkowski inequality up to a universal constant factor. 


\begin{theorem}[Milman's reverse Brunn-Minkowski inequality] \label{RBMIneq}
For any two convex bodies $K_1,K_2$ in ${\mathbb R}^n$, there exist linear volume preserving
transformations $T_{K_i}$ $(i=1,2)$, such that  for $\widetilde K_i = T_{K_i}(K_i)$ one has
$$ {\rm Vol} (\widetilde K_1+ \widetilde K_2)^{1/n} \leq  C \left ( {\rm Vol}(\widetilde K_1)^{1/n} + {\rm Vol}(\widetilde K_2)^{1/n} \right),$$
for some absolute constant $C$.
\end{theorem}
We emphasize that the transformation $T_{K_i} (i=1,2)$ in Theorem~\ref{RBMIneq} depends solely on the body $K_i$, and not on the joint configuration of the bodies $K_1$ and $K_2$. For more details on the reverse Brunn-Minkowski inequality see~\cite{Milm1317,P1}.

\smallskip

We can now sketch the proof of Theorem~\ref{up-to-uni-constnat} (for  more details see~\cite{AMO}). Since every symplectic capacity is bounded above by the cylindrical capacity  $\overline c$, it is enough to prove the theorem for $\overline c$. For the sake of simplicity, we assume in what follows that $K$ is centrally symmetric, i.e., $K=-K$.  This assumption is not too restrictive, since by a classical result of Rogers and Shephard~\cite{RogShe} one has that ${\rm Vol}(K+(-K)) \leq 4^n {\rm Vol}(K)$. 
After adjusting Theorem~\ref{RBMIneq} to the symplectic context, one has that for
any convex body $K \subset {\mathbb R}^{2n}$, there exists a linear symplectomorphism $S \in {\rm Sp}(2n)$ such that $SK$ and $iSK$ satisfy
the reverse Brunn-Minkowski inequality, that is, the volume  ${\rm Vol}(SK + iSK)$ is less than some constant times ${\rm Vol}(K)$. 
Combining this with the properties of symplectic capacities and Corollary~\ref{COR:about-sym-bodies}, 
%
%
%
we conclude that 
$$  {\frac {\overline c(K)} {\overline c(B)}}  \leq  {\frac {\overline c(SK+iSK)} {\overline c(B)}}   \leq A \, \left  (   {\frac {{\rm Vol}(SK+iSK)} {{\rm Vol}(B)}} \right )^{\frac 1 n}  \leq  A'  \, \left (   {\frac {{\rm Vol}(K)} {{\rm Vol}(B)}} \right )^{\frac 1 n} ,$$
for some universal constant $A'$, and thus Theorem~\ref{up-to-uni-constnat}  follows.

%
%
%

%
\smallskip

%
%
%

%
%
%
%
%

In the next section we will show a surprising connection between Viterbo's volume-capacity conjecture and a seemingly remote open conjecture from the field of convex geometric analysis: 
the Mahler conjecture on the volume product of centrally symmetric convex bodies.


\section{ A Symplectic View on Mahler's Conjecture} \label{SEC:MAHLER}

Let $(X,\| \cdot \|)$ be an $n$-dimensional normed space and let
$(X^*,\| \cdot \|^*)$  be its dual space. Note that the product space $X
\times X^*$ carries a canonical symplectic structure, given by the
skew-symmetric bilinear form $\omega \bigl ( (x,\xi),(x',\xi') \bigr
) = \xi(x')-\xi'(x)$, and a canonical volume form, the {\it
Liouville} volume, given by $ \omega^n/n!$. A fundamental question
in the field of convex geometry, raised by Mahler in~\cite{Ma}, is to find  upper and lower bounds for the
Liouville volume of $B \times B^{\circ} \subset X \times X^*$, where
$B$ and $B^{\circ}$ are the unit balls of $X$ and $X^*$,
respectively. In what follows we shall denote this volume by
$\nu(X)$. 
The quantity $\nu(X)$ is an affine invariant of $X$, i.e. it
is invariant under invertible linear transformations. We remark that
in the context of convex geometry $\nu(X)$ is also known as the
{\it Mahler volume} or the {\it volume product} of $X$.

\smallskip

The Blaschke-Santal\'o inequality asserts that the maximum of
$\nu(X)$ is attained if and only if $X$ is a Euclidean space. This
was proved by Blaschke~\cite{Bl} for dimensions two and three, and
generalized by Santal\'o~\cite{Sa} to higher dimensions. 
The following sharp lower bound for $\nu(X)$ was
conjectured by Mahler~\cite{Ma} in 1939:


\begin{conjecture}[Mahler's volume product conjecture]\label{Mahler-conj}
For any
$n$-dimensional normed space $X$ one has $\nu(X) \geq 4^n/n!.$ 
\end{conjecture}
%
%

The conjecture has been verified by Mahler~\cite{Ma} in the
two-dimensional case. In higher dimensions it is proved only in
a few special cases (see e.g.,~\cite{GMR,Kim,ME1,NPRZ, R1,R2, RSW,SR, St}). 
%
%
A major breakthrough towards answering Mahler's conjecture is a result 
due to Bourgain and Milman~\cite{BM}, who used sophisticated tools from functional analysis to
show that the conjecture holds asymptotically, i.e., up to a factor
$\gamma^n$, where $\gamma$ is a universal constant. 
This result has been re-proved later on, with entirely different methods, by Kuperberg~\cite{Ku}, using differential geometry, 
and independently by Nazarov~\cite{Naz},  using the theory of functions of several complex variables. A new proof using simpler asymptotic geometric analysis tools has been recently discovered by Giannopoulos, Paouris, and Vritsiou~\cite{GPV}. 
The best known constant today, $\gamma = \pi/4$, is due to Kuperberg~\cite{Ku}.


\smallskip

Despite great efforts to deal with the general case, a proof of Mahler's conjecture has been insistently elusive so far, and is currently the subject of intensive research.
A possible reason for this, as pointed out for example by Tao in~\cite{Tao}, is that, in contrast with the above mentioned Blaschke-Santal\'o inequality, the equality case
in Mahler's conjecture, which is obtained for example for the space
$l^n_{\infty}$ of bounded sequences with the standard maximum norm,
is not unique, and there are in fact many distinct extremizers for the (conjecturally) lower bound of $\nu(X)$ (see, e.g., the discussion in~\cite{Tao}).
This practically renders impossible any proof based on currently known optimisation techniques, and a radically different approach seems to be needed. 
%
%


\smallskip

We refer the reader to  Section~\ref{SEC:OQ} below for further discussion on the characterization of the equality case of Mahler's conjecture,  and its possible connection with symplectic geometry.

\smallskip

In a recent work with S. Artstein-Avidan and R. Karasev~\cite{AKO}, we combined  tools from symplectic geometry, convex analysis, and the theory of mathematical billiards, and  established a close relationship between 
Mahler's conjecture and Viterbo's volume-capacity conjecture. More precisely, we proved in~\cite{AKO} that
\begin{theorem} \label{THM-V-M} Viterbo's volume-capacity conjecture implies Mahler's conjecture.
\end{theorem}
In fact, it follows from our proof that Mahler's conjecture is  equivalent to a special case of Viterbo's  conjecture, where the latter is restricted to the Ekeland-Hofer-Zehnder symplectic capacity, and to domains in the classical phase space of the form $\Sigma \times \Sigma^{\circ} \subset {\mathbb R}^{2n} = {\mathbb R}^n_q \times {\mathbb R}^n_p$   (for more details see~\cite{AKO}, and in particular Remark 1.9 ibid.). Here, $\Sigma \subset {\mathbb R}^n_q$ is a centrally symmetric convex body, the space ${\mathbb R}^n_p$ is identified with the dual space $({\mathbb R}^n_q)^*$, and $$\Sigma^{\circ} = \{ p \in  {\mathbb R}^n_p \, | \, p(q)\leq 1 \ {\rm for \ every } \ q \in \Sigma 
 \}$$

Theorem~\ref{THM-V-M} is a direct consequence of 
the following result proven in~\cite{AKO}.

\begin{theorem} \label{THM:CAP=4} There exists a symplectic capacity $c$  such that $ c(\Sigma \times \Sigma^{\circ})=4$ for every centrally symmetric convex body $\Sigma \subset {\mathbb R}^n_q$.
\end{theorem}

With Theorem~\ref{THM:CAP=4} at our disposal, it is not difficult to derive Theorem~\ref{THM-V-M}.
\begin{proof}[{\bf Proof of Theorem~\ref{THM-V-M}}]
Assume that Viterbo's volume-capacity conjecture holds. From Theorem~\ref{THM:CAP=4} it follows that there exists a symplectic capacity $c$ such that for every centrally symmetric convex body $\Sigma \subset {\mathbb R}^n_q$ one has
$$ {\frac {4^n} {\pi^n}} = {\frac { c^n(\Sigma \times \Sigma^{\circ})} {\pi^n}}
 \leq   {\frac {{\rm Vol}(\Sigma \times \Sigma^{\circ}) } {{\rm Vol}(B^{2n})} }  =  {\frac {n! \, {\rm Vol}(\Sigma \times \Sigma^{\circ})} {{\pi^n}}
},
$$
which  is exactly the bound for  ${\rm Vol}(\Sigma \times
\Sigma^{\circ})$ required by Mahler's conjecture.
\end{proof}

In the rest of this section we sketch the proof of Theorem~\ref{THM:CAP=4} (see~\cite{AKO} for a detailed exposition). %
We remark that an alternative proof, based on an approach to billiard dynamics developed in~\cite{BB}, 
was recently given in~\cite{ABKS}. We start with recalling the definition of 
the Ekeland-Hofer-Zehnder capacity, which is the symplectic capacity that appears in Theorem~\ref{THM:CAP=4}.

%
%

\smallskip 

The restriction of the standard symplectic form
$\omega=dq \wedge dp$ to a smooth closed connected 
hypersurface ${\mathcal S}
\subset {\mathbb R}^{2n}$ defines a 1-dimensional subbundle
${\rm ker}(\omega | {\mathcal S})$, whose integral curves comprise the
characteristic foliation of ${\mathcal S}$. In other words, a {\it closed
characteristic}  of  ${\mathcal S}$ is an embedded circle
in  ${\mathcal S}$ tangent to the canonical 
line bundle
\begin{equation*} {\mathfrak S}_{{\mathcal S}} = \{(x,\xi) \in T
{\mathcal S} \ | \ \omega(\xi,\eta) = 0 \ {\rm for \ all} \ \eta \in T_x
{\mathcal S} \}. \end{equation*}
Recall that the symplectic action 
of a closed curve $\gamma$ 
is defined by $A(\gamma) = \int_{\gamma} \lambda$,
where  $\lambda =pdq$ is the Liouville 1-form. 
The action spectrum of ${\mathcal S}$ is 
\begin{equation*}  {\cal L}({\mathcal S}) = \left \{ \, | \, {A}({\gamma}) \,  | \, ;
\, \gamma \ {\rm closed \ characteristic \ on} \  {\mathcal S}
\right \}.\end{equation*}

The following theorem, which is a combination of results from~\cite{EH} and~\cite{HZ}, states that on the class of convex domains in ${\mathbb R}^{2n}$,  the  Ekeland-Hofer capacity $c_{_{\rm EH}}$ and Hofer-Zehnder capacity $c_{_{\rm HZ}}$ coincide, 
and are given by the minimal action over all closed characteristics on the boundary of the corresponding convex body. 
%
%
\begin{theorem} \label{Cap_on_covex_sets} Let $ K \subseteq {\mathbb
R}^{2n}$ be a convex bounded domain with smooth boundary. 
Then there exists at least one closed characteristic $\widetilde \gamma
\subset \partial K$ satisfying
\begin{equation*} 
 c_{_{\rm EH}}(K) = c_{_{\rm HZ}}(K)= { A}(\widetilde \gamma) =  \min {\cal
L}( \partial K). \end{equation*}
\end{theorem}

We remark that although the above definition of closed characteristics, as well as Theorem~\ref{Cap_on_covex_sets}, 
were given only for the class of convex bodies with smooth boundary, they can naturally be generalized 
to the class of convex sets in ${\mathbb R}^{2n}$ with non-empty interior  (see~\cite{AAO1}). In what follows, we
refer to the coinciding Ekeland-Hofer and Hofer-Zehnder capacities on this class as the Ekeland-Hofer-Zehnder capacity. 

\smallskip

We turn now to show that for every centrally symmetric convex body $\Sigma \subset {\mathbb R}^n_q$, the Ekeland--Hofer--Zehnder capacity satisfies $c_{_{\rm EHZ}} (\Sigma \times \Sigma^{\circ})=4$.
For this purpose, we now switch gears and turn to mathematical billiards in Minkowski geometry. 

\smallskip

It is  folklore to people in the field that billiard flow can be treated, roughly speaking, as the limiting case of geodesic flow on a boundaryless manifold. 
Indeed, 
%
let $\Omega$ be a smooth plane billiard table, and consider its ``thickening", i.e. an infinitely thin three
dimensional body whose boundary $\Gamma$ is obtained by pasting two copies of $\Omega$ along their boundaries
and smoothing the edge. Thus, a billiard trajectory in $\Omega$ can be viewed as a geodesic line on
the boundary of $\Gamma$, that goes from one copy of $\Omega$ to another each time the billiard ball bounces off the boundary. 
The main technical difficulties  with this strategy is the rigorous treatment of the limiting process, and the analysis involved with the dynamics near the  boundary. 
One approach to  billiard dynamics and the existence question of periodic trajectories is an  approximation scheme which uses a certain ``penalization method"  developed by Benci and  Giannoni in~\cite{BG} (cf.~\cite{AM,IrieKei2}).
In what follows we present an alternative approach, and use characteristic foliation on singular  convex hypersurfaces in ${\mathbb R}^{2n}$ (see e.g.,~\cite{Cl,Eke,Kun}) to describe Finsler type billiards for convex domains in the configuration space ${\mathbb R}^n_q$.
The main advantage of this approach is that it allows one to use the natural one-to-one correspondence between the geodesic flow on a manifold  and the characteristic foliation on its unit cotangent bundle, and thus provides a natural  ``symplectic setup"  in which one can use tools 
such as Theorem~\ref{Cap_on_covex_sets} above in the context of billiard dynamics.  In particular, we show that the Ekeland-Hofer-Zehnder capacity of certain Lagrangian product configurations ${\mathcal K} \times {\mathcal T}$ in the classical phase space ${\mathbb R}^{2n}$  is the length of the shortest periodic  {\it  ${\mathcal T}$-billiard trajectory in ${\mathcal K}$} (see e.g.,~\cite{AAO1,V}), which we turn now to describe.

\smallskip

The general study of billiard dynamics in Finsler  and Minkowski
geometries was initiated by Gutkin and Tabachnikov in~\cite{GT}. From the point of view of geometric optics, Minkowski billiard
trajectories describe the propagation of light in a homogeneous
anisotropic medium that contains perfectly reflecting mirrors. 
Below, we focus on the special case of Minkowski billiards in a smooth convex
body  ${\mathcal K} \subset {\mathbb R}^n_q$. 
We equip ${\mathcal K}$
with a metric given by a certain norm $\| \cdot \|$, and
consider billiards in ${\mathcal K}$ with respect to the geometry induced by $\| \cdot \|$.
More precisely, let ${\mathcal K} \subset {\mathbb R}^n_q$, and ${\mathcal T}  \subset {\mathbb R}^n_p$ be two convex bodies with smooth boundary, and consider the 
unit cotangent bundle
\begin{equation*}
U_{\mathcal  T}^*{\mathcal K} := {\mathcal K}  \times {\mathcal T} = \{ (q,p) \, | \, q \in {\mathcal K}, \ {\rm and} \
g_{\mathcal T} (p)  \leq 1 \} \subset  T^* {\mathbb R}^n_q  = {\mathbb R}^{n}_q
\times  {\mathbb R}^{n}_p. \end{equation*} Here $g_{\mathcal T}$ is the gauge function 
$g_{\mathcal T}(x) = \inf \{r  \, | \, x \in r {\mathcal T} \}$. 
When ${\mathcal T}=-{\mathcal T}$ is centrally symmetric  
  one has $g_{\mathcal T}(x) = \|x\|_{\mathcal T}$.
For $p \in \partial {\mathcal T}$, the gradient vector $\nabla g_{\mathcal T}(p)$ is the outer normal to $\partial {\mathcal T}$ at the point $p$,  and is naturally considered to be in $\mathbb R^n_q = (\mathbb R^n_p)^* $.

\smallskip

Motivated by the classical correspondence between 
geodesics in a Riemannian manifold  and 
characteristics of its unit cotangent bundle, 
we define $({\mathcal K},{\mathcal T})$-billiard trajectories  to be characteristics in 
$U_{\mathcal  T}^*{\mathcal K}$ such that  their projections to ${\mathbb R}^n_q$ are  
closed billiard trajectories in ${\mathcal K}$ with a bouncing rule
that is determined by the geometry induced from the body ${\mathcal T}$; and vice versa, the projections to ${\mathbb R}^n_p$ are  
closed billiard trajectories in ${\mathcal T}$ with a bouncing rule
that is determined by 
${\mathcal K}$.
More precisely, 
when we follow the  vector fields of the dynamics, 
we move in ${\mathcal K} \times \partial {\mathcal T}$ from $(q_0,p_0)$ to
$(q_1,p_0) \in \partial {\mathcal K} \times \partial {\mathcal T}$ following the inner normal to
$\partial {\mathcal T}$ at $p_0$. When we hit the boundary $\partial {\mathcal K}$ at the
point $q_1$, the vector field  changes, and we start to move in
$\partial {\mathcal K}  \times {\mathcal T}$ from $(q_1,p_0)$ to $(q_1,p_1) \in \partial {\mathcal K} \times \partial {\mathcal T}$ following the outer
 normal to $\partial {\mathcal K}$ at the point $q_1$.  Next, we move from
$(q_1,p_1)$ to $(q_2,p_1)$ following the opposite of the normal to
$\partial {\mathcal T}$ at $p_1$, and so on and so forth (see
Figure $1$). 
It is not hard to check that when one of the bodies, say ${\mathcal T}$, is a Euclidean
ball, then when considering the projection to ${\mathbb R}^{n}_q$, the bouncing rule described above is the
classical one (i.e., equal impact and
reflection angles). 
Hence, the above 
reflection law is a
natural variation of the classical one 
when the Euclidean structure on ${\mathbb R}^n_q$
is replaced by the metric induced by the norm $\| \cdot \|_{{\mathcal T}}$.
We continue with a more precise definition. 
%

\begin{definition} \label{def-of-periodic-traj} Given two smooth convex bodies ${\mathcal K} \subset {\mathbb R}^n_q$ and ${\mathcal T} \subset {\mathbb R}^n_p$. 
A closed $({\mathcal K},{\mathcal T})$-billiard trajectory is the image of a piecewise smooth
map $\gamma \colon S^1 \rightarrow \partial ({\mathcal K} \times {\mathcal T}) $
such that for every  $t \notin {\mathcal B}_{\gamma}:= \{ t
\in S^1 \, | \, \gamma(t) \in \partial {\mathcal K}  \times \partial {\mathcal T} \}$ one has
\begin{equation*}
\dot \gamma(t) = d \, {\mathfrak X}(\gamma(t)),  \end{equation*} for some positive
 constant $d$ and the vector field ${\mathfrak X}$  given by
\begin{equation*}
{\mathfrak X}(q,p) = \left\{
\begin{array}{ll}
(-\nabla g_{\mathcal T}(p) ,0), &   (q,p) \in int({\mathcal K}) \times \partial {\mathcal T},\\
(0,\nabla g_{\mathcal K}(q)), & (q,p) \in \partial {\mathcal K} \times int({\mathcal T}).
\end{array} \right.
\end{equation*}
Moreover, for any $t \in {\mathcal B}_{\gamma}$, the left and right
derivatives of $\gamma(t)$ exist, and
\begin{equation*} \label{eq-the-cone}
\dot \gamma^{\pm}(t) \in \{   \alpha (-\nabla g_{\mathcal T}(p) ,0) + \beta
(0,\nabla g_{\mathcal K}(q))    \ | \ \alpha,\beta \geq 0,  \ (\alpha, \beta) \neq (0,0) \}.
\end{equation*}
\end{definition}
Although  in Definition~\ref{def-of-periodic-traj} 
there is a natural symmetry between the  bodies ${\mathcal K}$ and ${\mathcal T}$, 
in what follows  we shall  assume that  ${\mathcal K}$ 
plays the role of the billiard table, while  ${\mathcal T}$ induces the geometry that governs the billiard dynamics in ${\mathcal K}$.
We will use the following terminology: 
for a $({\mathcal K},{\mathcal T})$-billiard trajectory $\gamma$, the curve $\pi_q(\gamma)$,  where $\pi_q \colon {\mathbb R}^{2n} \rightarrow {\mathbb R}^n_q$ is the projection of $\gamma$ to the configuration space, 
shall be called a {\it  ${\mathcal T}$-billiard trajectory in ${\mathcal K}$}. Moreover,  similarly to the Euclidean case, one can check that ${\mathcal T}$-billiard trajectories in ${\mathcal K}$   
correspond to critical points of a length functional defined on the $j$-fold cross product of the boundary $\partial {\mathcal K}$, where the distances between two consecutive  points are measured with respect to  
 the support function $h_{\mathcal T}$, where $h_{\mathcal T}(u) = \sup \{ \langle x,u \rangle \, ; \, x \in {\mathcal T } \}$.  

\begin{definition} 
A closed $({\mathcal K},{\mathcal T})$-billiard trajectory $\gamma$ is said to be {\it proper}
if the set ${\mathcal B}_{\gamma}$ is finite, i.e., $\gamma$ is a  
broken bicharacteristic that enters and instantly exits the boundary
$\partial {\mathcal K} \times \partial {\mathcal T}$ at the reflection points.
In the case where ${\mathcal B}_{\gamma} = S^1$, i.e., $\gamma$ is travelling
solely along the boundary $\partial {\mathcal K}  \times \partial {\mathcal T}$,
 we say that $\gamma$ is a {\it gliding trajectory}.
\end{definition}

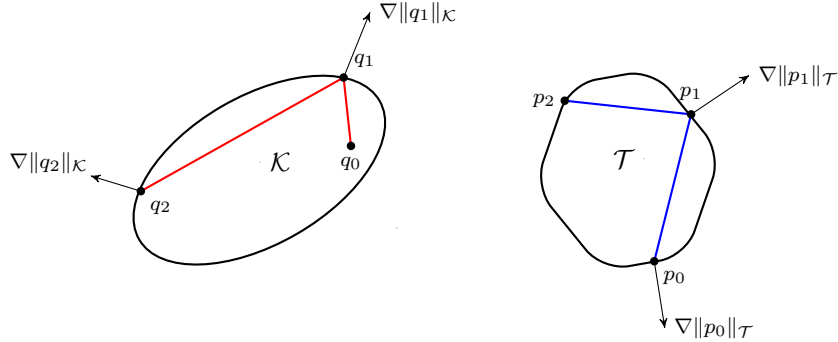
\begin{figure} 
\begin{center}
\begin{tikzpicture}[scale=0.7]

 \draw[important line][rotate=30] (0,0) ellipse (75pt and 40pt);

 \path coordinate (w1) at (2.1,4*0.75) coordinate (q0) at
 (-4.5*0.5,-2*0.2) coordinate (q1) at (1.6,4*0.44) coordinate (q2) at
 (1.74,+0.46) coordinate (w2) at (2.6,-1.1) coordinate (w3) at (-3.2,-0.11) coordinate (K) at
 (0,0.15);

\draw[red] [important line] (q0) -- (q1); \draw[->] (q1) -- (w1);
\draw[red] [important line] (q1) -- (q2); 
\draw[->] (q0) -- (w3);

\filldraw [black]
  (w3) circle (0pt) node[above left=-2.5pt] {{\footnotesize $\nabla \|q_2\|_{\mathcal K}$}}
    (w1) circle (0pt) node[right] {{\footnotesize $\nabla \|q_1\|_{\mathcal K}$}}
    (w2) circle (0pt) 
     (q0) circle (2pt) node[below right] {{\footnotesize $q_2$}}
      (q1) circle (2pt) node[above right=0.5pt] {{\footnotesize $q_1$}}
       (q2) circle (2pt) node[below=0.5pt] {{\footnotesize $q_0$}}
        (K) circle (0pt) node[right=0.5pt] {${\mathcal K}$};

       \begin{scope}[xshift=7cm]

 \draw[important line][rounded corners=10pt][rotate=10] (1.8,0) --
 (0.8,1.8)-- (-0.8,1.8)--  (-1.8,0)--  (-0.8,-1.8) -- (0.8,-1.8) --
 cycle;

 \path coordinate (p1) at (0.5,-4*0.433) coordinate (np0) at
 (2.3,2*0.9) coordinate (p0) at (1.2,2*0.53) coordinate (np1) at
 (0.7,-3) coordinate (p2) at (-1.2,2*0.66) coordinate (D) at
 (0.3,0.22);

 \draw[<-] (np0) node[right] {{\footnotesize $\nabla \|p_1
 \|_{\mathcal T}$}} -- (p0);
 \draw[<-] (np1) node[right] {{\footnotesize $\nabla \|p_0
 \|_{\mathcal T}$}} -- (p1);

 \draw[blue][important line] (p0) -- (p1);
 \draw[blue][important  line] (p0) -- (p2);

  \filldraw [black]
       (p1) circle (2pt) node[below right] {{\footnotesize $p_0$}}
         (p2) circle (2pt) node[left] {{\footnotesize $p_2$}}
         (p0) circle (2pt) node[above=2pt] {{\footnotesize $p_1$}}
          (D) circle (0pt) node[left] {${\mathcal T}$};
 \end{scope}
 \end{tikzpicture}

 \caption{A proper $({\mathcal K},{\mathcal T})$-Billiard trajectory.} 
 \end{center}
 \end{figure}

The following theorem was proved in~\cite{AAO1}.
\begin{theorem} \label{Main-Theorem-From-AAO1}
Let ${\mathcal K}  \subset {\mathbb R}_q^n$, ${\mathcal T}  \subset {\mathbb R}_p^n$ be
two smooth  convex bodies.  Then, every $({\mathcal K},{\mathcal T})$-billiard trajectory is either a proper trajectory, or a gliding one.
Moreover, the Ekeland-Hofer-Zehnder capacity $c_{_{\rm EHZ}}({\mathcal K} \times {\mathcal T})$, of the Lagrangian product ${\mathcal K} \times {\mathcal T}$, is the length of the shortest periodic ${\mathcal T}$-billiard trajectory in ${\mathcal K}$, measured with respect to the support function $h_{\mathcal T}$. 
\end{theorem}

This theorem provides an effective way to estimate (and sometimes compute) the Ekeland-Hofer-Zehnder capacity of Lagrangian product configurations in the phase space. For example, in~\cite{AKO} (see Remark 4.2 therein) we used elementary tools from convex geometry to show that for centrally symmetric convex bodies, 
the shortest ${\mathcal T}$-billiard trajectory in ${\mathcal K}$ is a 2-periodic trajectory connecting a tangency point $q_0$ of ${\mathcal K}$ and a homotetic copy of $ {\mathcal T}^{\circ}$  to $-q_0$  (see Figure 2).  This result extends a previous result by Ghomi~\cite{Gh} for Euclidean billiards. 
In both cases, the main difficulty in the proof is to show 
that the above mentioned 2-periodic trajectory is indeed the shortest one. 
With this geometric observation at our disposal, we proved in~\cite{AKO}  the following result: denote by 
${\rm inrad}_{{\mathcal T}}({\mathcal K})  =  \max \{r \, | \, r{\mathcal T} \subset {\mathcal K} \}$. 

\begin{theorem} \label{Main-Theorem-From-AKO}
If ${\mathcal K}  \subset {\mathbb R}^n_q$, ${\mathcal T} \subset {\mathbb R}^n_p$ are centrally symmetric convex bodies, then
$$c_{_{\rm EHZ}}({\mathcal K}  \times {\mathcal T}) =  \overline c({\mathcal K}  \times {\mathcal T}) =  4 \, {\rm inrad}_{{\mathcal T}^{\circ}}({\mathcal K})$$
\end{theorem}
Note that Theorem~\ref{Main-Theorem-From-AKO} immediately implies Theorem~\ref{THM:CAP=4}  above, which in turn implies Theorem~\ref{THM-V-M}.
Thus, we have shown that Mahler's conjecture follows 
from a special case of Viterbo's conjecture.  
In fact, it follows immediately from the proof of Theorem~\ref{THM-V-M}  that Mahler's conjecture is equivalent to Viterbo's conjecture when the latter is restricted to the Ekeland-Hofer-Zehnder capacity, and to convex domains of the form $\Sigma \times \Sigma^{\circ}$, where $\Sigma \subset {\mathbb R}^n_q$ is a centrally symmetric convex body.  We hope that further pursuing this line of research will lead to a  breakthrough in understanding both conjectures.

\begin{figure} 
\begin{center}
  \begin{tikzpicture}[scale=0.8] \label{Fig1}

\path coordinate (q1) at (-0.3,1.9) coordinate (q2) at (0.55,0.32) coordinate (q3) at (-0.3,1) coordinate (q4) at (0.3,-1) coordinate (q5) at (-0.80,1.9) coordinate (q6) at (0.8,-1.9)   ; 
\filldraw [black]
 (q3) circle (1pt) node[above] {{\footnotesize $\widetilde q$}}
 (q4) circle (1pt) node[below left=-0.7pt] {{\footnotesize $-\widetilde q$}}
 (q1) circle (0pt) node[below right=4pt] {{\footnotesize ${\mathcal K}$}}
(q2) circle (0pt) node[above right=-3.2pt] {{\footnotesize ${ r} \, {\mathcal T}^{\circ}$}};

\draw[dashed] (q3) -- (q4); 
 \draw[thick,->,blue] (q3)--(q5) node[left, black ] {${\scriptstyle  \nabla \| \widetilde q \|_{\mathcal K}}$};
  \draw[thick,->,blue] (q4)--(q6) node[right, black ] {${\scriptstyle \nabla \| - \widetilde q \|_{\mathcal K}}$};

  \draw[blue,rotate=30] (0,0) ellipse (2.2cm and 1cm);
  \draw[red,rotate=-80] (0,0) ellipse (1.05cm and 0.55cm);
      \begin{scope}[xshift=7cm]

\path coordinate (p1) at (-0.0,1.65) coordinate (p2) at (0.50,0.23) coordinate (p3) at (-0.75,1.32) coordinate (p4) at (0.75,-1.32) coordinate (p5) at (-0.7*1.4,1.2*1.8) coordinate (p6) at (0.7*1.4,-1.2*1.8)  ; 
\filldraw [black]
 (p2) circle (0pt) node[right=4pt] {{\footnotesize ${ r} \, {\mathcal K}^{\circ}$}}
(p1) circle (0pt) node[right] {{\footnotesize ${\mathcal T}$}};

\filldraw [black]
 (p3) circle (1pt) node[above right=-1pt] {{\footnotesize ${ } \widetilde p$}}
 (p4) circle (1pt) node[below left=-1pt] {{\footnotesize $-{ } \widetilde p$}};

\draw[dashed] (p3) -- (p4); 
 \draw[thick,->,red] (p3)--(p5) node[left, black ] {${\scriptstyle { \nabla \| \widetilde p \|_{\mathcal T}}}$};
  \draw[thick,->,red] (p4)--(p6) node[right, black ] {${\scriptstyle  \nabla \| - \widetilde p \|_{\mathcal T} }$};

  \draw[red,rotate=-20] (0,0) ellipse (1.8cm and 1.3cm);
  \draw[blue,rotate=-55] (0,0) ellipse (1.55cm and 0.7cm);
\end{scope}
\end{tikzpicture}

\caption{\it ${\mathcal T}$-billiard trajectory in ${\mathcal K}$ of length $4 \, {\rm inrad}_{{\mathcal T}^{\circ}}({\mathcal K})$.} 
\end{center}
\end{figure}
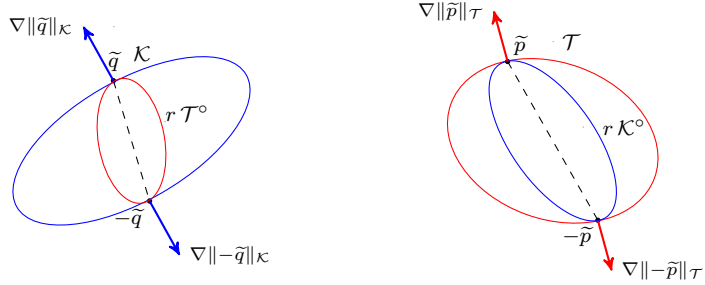

\subsection{Bounds on the length of the shortest billiard trajectory}

Going somehow in the opposite direction,  
one can also use  the theory of symplectic capacities to provide
several bounds and inequalities for the length of the shortest periodic
billiard trajectory in a smooth convex body in ${\mathbb R}^n$. 
In~\cite{AAO1} we prove the following theorem, which for the sake of simplicity we state only for the case of Euclidean billiards (for several other related results see~\cite{ABKS,AM,BB,Gh,IrirKei1, IrieKei2,V}). 
%
\begin{theorem} \label{THM-Billiard} Let  $K \subset {\mathbb R}^n$ be a smooth convex body, and let $\xi(K)$ denote the length of the shortest periodic billiard trajectory in $K$. Then, 
\begin{enumerate}
\item[(i)] $\xi(K_1) \leq \xi(K_2)$, for any convex domains $K_1 \subseteq K_2 \subseteq {\mathbb R}^n$ (monotonicity);
\item[(ii)] $\xi(K) \leq C \sqrt n \, {\rm Vol}(K)^{\frac 1 n},$ for some universal constant $C>0$;
\item[(iii)] $4  {\rm inrad}(K) \leq \xi(K) \leq 2 (n+1) {\rm inrad}(K)$;
\item[(iv)] $\xi(K_1 + K_2) \geq \xi(K_1) + \xi(K_2)$ (Brunn-Minkowski type inequality). 
\end{enumerate}
\end{theorem} 
We remark that the inequality $4 {\rm inrad}(K) \leq \xi(K)$ in ${\it (iii)}$ above was proved already in~\cite{Gh}, the monotonicity property was  
well known to experts in the field (although it has not been addressed in the literature to the best of our knowledge), and  all the results in Theorem~\ref{THM-Billiard} were later recovered and generalized by different methods (see~\cite{ABKS,IrirKei1, IrieKei2}).
Moreover, in light of the ``classical versus quantum" relation between the length spectrum in Riemannian geometry
and the Laplace spectrum, via trace formulae and Poisson relations, Theorem~\ref{THM-Billiard} 
can be viewed as
a classical counterpart of some  well-known results  for the first Laplace eigenvalue 
on convex domains. It is interesting to note that, to the best of the author's knowledge, the exact value of the constant $C$ in part  ${\it (ii)}$ of Theorem~\ref{THM-Billiard}   is unknown already in the two-dimensional case.

\section{The Uniqueness of Hofer's Metric} \label{SEC:HOFER}

One of the most striking facts regarding 
the group of Hamiltonian diffeomorphisms associated with  a symplectic manifold is that it can be
equipped with an intrinsic geometry given by a bi-invariant Finsler metric known as
Hofer's metric~\cite{H}. In contrast to the case of finite-dimensional Lie groups, the existence of such
a metric on an infinite-dimensional group of transformations is highly unusual due to
the lack of local  compactness.
Hofer's metric is exceptionally important for at least two reasons: first, Hofer showed in~\cite{H} that this metric 
gives rise to an important symplectic capacity known as  ``displacement energy", which turns out to have many different applications in symplectic topology and Hamiltonian dynamics (see e.g.,~\cite{Chek,H,HZ, L,LM,P,P1}). Second, it provides a certain geometric intuition for
the understanding of the long-time behaviour of Hamiltonian dynamical systems.

\smallskip

In~\cite{EliP}, Eliashberg and Polterovich initiated a discussion on the uniqueness of Hofer's metric
(cf.~\cite{Eli,P1}). They asked whether for a closed symplectic manifold $(M,\omega)$, Hofer's metric is the only 
 bi-invariant Finsler metric on the group  of Hamiltonian diffeomorphisms. 
In this section we explain  (following~\cite{OW} and~\cite{BO}) how tools from  classical functional analysis and the theory of normed function spaces  can be used to  positively  answer this question, and show that up to equivalence of metrics, Hofer's metric is unique. 
For this purpose, we now turn to more precise
formulations.



%

\smallskip

Let $(M,\omega)$ be a closed $2n$-dimensional symplectic manifold,
and denote by $C^{\infty}_0(M)$ the space of smooth functions that
are zero-mean normalized with respect to the canonical volume form
$\omega^n$. With every smooth time-dependent Hamiltonian function $H: M \times [0,1] \rightarrow {\mathbb
R}$, 
one associates a  vector field $X_{H_t}$ 
via the equation $i_{X_{H_t}} \omega = - dH_t$, where $ H_t(x) =
H(t,x)$. The flow
of  $X_{H_t}$ 
is denoted by $\phi_H^t$ and is defined for all $t \in [0,1]$.
The 
group of Hamiltonian
diffeomorphisms consists of all the time-one maps of such
Hamiltonian flows, i.e.,
$$ {\rm Ham}(M,\omega) = \{ \phi_H^1 \ |  \ \phi_H^t \
{\rm is \ a \ Hamiltonian \ flow  } \}.$$ When  equipped with the
standard $C^{\infty}$-topology, the group 
${\rm Ham}(M,\omega)$ 
is an infinite-dimensional Fr\'echet Lie group. Its Lie algebra, denoted here by 
${\cal A}$, can be naturally identified with the space of normalized smooth functions $C^{\infty}_0(M)$.
Moreover, the adjoint action of Ham$(M,\omega)$ on ${\cal A}$ is the
standard action of diffeomorphisms on functions, i.e., ${\rm Ad}_\phi f = f
\circ \phi^{-1}$, for every $f \in {\cal A}$  and $\phi \in$
Ham$(M,\omega)$. For more details on the group of  Hamiltonian
diffeomorphisms see e.g.,~\cite{HZ, McSal,P1}. 

\smallskip

Next, we define a Finsler pseudo-distance on
Ham$(M,\omega)$. Given any pseudo-norm $\| \cdot \|$ on 
${\cal A}$, we define the length of a path $\alpha : [0,1]
\rightarrow {\rm Ham}(M,\omega)$ as
$$ {\rm length}\{ {\alpha}\} = \int_0^1 \| \dot \alpha \| dt =
\int_0^1 \| H_t \| dt ,$$ where $H_t(x)=H(t,x)$ is the unique
normalized Hamiltonian function generating the path $\alpha$. Here
$H$ is said to be normalized if $\int_M H_t \omega^n=0$ for every
$t\in [0,1]$. The distance between two Hamiltonian diffeomorphisms
is given by $$ d(\psi,\varphi) := \inf {\rm length} { \{ \alpha \}
},$$ where the infimum is taken over all Hamiltonian paths $\alpha$
connecting $\psi$ and $\varphi$. It is not hard to check that $d$ is
non-negative, symmetric, and satisfies the triangle inequality.
Moreover, any pseudo-norm on the Lie algebra ${\cal A}$ that is invariant under the
adjoint action yields a bi-invariant pseudo-distance function on ${\rm Ham} (M,\omega)$, i.e.,
 $d(\psi,\phi) = d(\theta \, \psi,\theta \, \phi) = d(\psi \, \theta ,\phi \, \theta)$,
 for every  $\psi,  \phi,  \theta \in {\rm Ham} (M,\omega)$.

\smallskip

{\bf From here forth we  deal solely with such pseudo-norms
and we  refer to $d$
as the pseudo-distance generated by the pseudo-norm $\| \cdot \|$. }

\smallskip

We remark in passing that
a fruitful study of right-invariant Finsler metrics on
Ham$(M,\omega)$, motivated in part by applications to hydrodynamics,
was initiated  by Arnold~\cite{Ar}. 
In addition, non-Finslerian bi-invariant metrics on Ham$(M,\omega)$ have been
intensively studied in the realm of symplectic geometry, starting
with the works of Viterbo~\cite{V1}, Schwarz~\cite{Sch}, and
Oh~\cite{Oh}, and followed by many others. 
\begin{remark} \label{Rmk-about-continuity} {\rm
When one studies  geometric properties of the group of
Hamiltonian diffeomorphisms, it is convenient to consider smooth
paths $ [0,1] \rightarrow {\rm Ham}(M,\omega) $, among which those
that start at the identity correspond to smooth Hamiltonian flows.
Moreover, for a given Finsler pseudo-metric on $ {\rm Ham}(M,\omega)$, a natural geometric assumption is that every
smooth path $ [0,1] \rightarrow {\rm Ham}(M,\omega) $ has finite
length. As it turns out,  the latter 
assumption is equivalent to the
continuity of the pseudo-norm on ${\cal A}$ corresponding to the
pseudo-Finsler metric in the $ C^{\infty} $-topology (see~\cite{BO}).
Thus, in what follows 
we shall mainly consider such pseudo-norms.}
 \end{remark}

It is highly non-trivial to check whether a distance function on the group of Hamiltonian diffeomorphisms 
generated by  a pseudo-norm is non-degenerate, that is, $d({\rm Id},\phi)
> 0$ for $\phi \neq {\rm Id}$. In fact, for closed
symplectic manifolds, a bi-invariant pseudo-metric $d$ on
Ham$(M,\omega)$ is either a genuine metric or identically zero. This
is an immediate corollary of a well-known theorem by
Banyaga~\cite{B}, which states that Ham$(M,\omega)$ is a simple
group, combined with the fact that the null-set $${\rm null}(d) = \{
\phi \in {\rm Ham}(M,\omega) \ | \ d({\rm Id},\phi) = 0 \}$$ is a normal
subgroup of Ham$(M,\omega)$. A  renowned result by
Hofer~\cite{H} states that the $L_{\infty}$-norm on ${\cal A}$ gives
rise to a genuine distance function on Ham$(M,\omega)$ known now as
Hofer's metric. This was  proved by Hofer for the case
of ${\mathbb R}^{2n}$, then generalized by Polterovich~\cite{P}, and
finally proven in full generality by Lalonde and McDuff~\cite{LM}.
In a sharp contrast to the above, Eliashberg and
Polterovich showed in~\cite{EliP} that for a closed symplectic manifold $(M,\omega)$ ons has
\begin{theorem}[Eliashberg and Polterovich]
For $1 \leq p < \infty$,  the
pseudo-distances on ${\rm Ham}(M,\omega)$ corresponding to the
$L_p$-norms on ${\cal A}$ vanish identically.
\end{theorem}

The following question was asked in~\cite{EliP} (cf.~\cite{Eli,P1}):

\begin{question} \label{Ques:inv-norms} What are the ${\rm Ham}(M,\omega)$-invariant norms on ${\cal
A}$, and which of them give rise to genuine bi-invariant metrics
on ${\rm Ham}(M,\omega)$?
\end{question}

It was observed in~\cite{BO} that  any pseudo-norm $\| \cdot \|$ 
on the space ${\mathcal A}$ can be turned into a Ham$(M,\omega)$-invariant pseudo-norm via a certain 
invariantization procedure $ \| f \| \mapsto \| f \|_{\rm inv}$. 
The idea behind this procedure is based on the notion of infimal convolution (or epi-sum), from convex analysis. 
Recall that the infimal convolution of two functions $f$ and $g$ on ${\mathbb R}^n$ is defined by $(f \square g)(z)  = \inf \{ f(x)+g(y) \, | \, z=x+y\}$. 
This operator has a simple geometric interpretation: the epigraph (i.e., the set of points lying on or above the graph) of the infimal convolution of two functions is the Minkowski sum of the epigraphs of those functions. 
The invariantization $\| \cdot \|_{\rm inv}$  of $\| \cdot \|$ is obtained by taking the orbit of $\|  \cdot \|$ under the group action, and consider the infimal convolution of the associated family of norms. More preciesly, define

$$ \| f \|_{\rm inv} =  \inf \Bigl \{  \sum \|\phi_i^*  f_i \|  \ ; \ f = \sum f_i, \ {\rm and \ } \phi_i \in {\rm Ham}(M,\omega)  \Bigr \}. $$
We remark that in the above definition of $\| f \|_{\rm inv}$ the sum
$\sum f_i$ is assumed to be finite. Note that $\| \cdot \|_{\rm inv}
\leq \| \cdot \|$.  Thus, if $\| \cdot \|$ is continuous in the
$C^{\infty}$-topology, then so is $\| \cdot \|_{\rm inv}$. Moreover, if
$\| \cdot \|'$ is a Ham$(M,\omega)$-invariant pseudo-norm, then:
$$ \| \cdot \|' \leq \| \cdot \| \Longrightarrow \| \cdot \|' \leq \| \cdot \|_{\rm inv}.$$
In particular, the above invariantization procedure provides a
plethora of Ham$(M,\omega)$-invariant genuine norms on ${\mathcal
A}$, e.g., by applying it to the $\| \cdot \|_{C^k}$-norms.

\smallskip

In~\cite{OW} we made a first step toward answering Question~\ref{Ques:inv-norms} using tools from the theory of normed spaces and functional analysis. More precisely, regarding the first part of Question~\ref{Ques:inv-norms}, we proved

\begin{theorem}[Ostrover and Wagner] \label{Ham-invariant-implies-measure-invarinat}
Let $ \| \cdot \|$ be a ${\rm Ham}(M,\omega)$-invariant norm on ${\cal
A}$ such that $\| \cdot \| \leq C \| \cdot \|_{\infty}$ for some
constant $C$. Then $\| \cdot \|$ is invariant under all measure
preserving diffeomorphisms of $M$.
\end{theorem}
In other words, any ${\rm Ham}(M,\omega)$-invariant  norm on  ${\cal A}$ that is bounded above by the $L_{\infty}$-norm, must also be invariant  under the much larger group of measure preserving diffeomorphisms. 
%
%
Theorem~\ref{Ham-invariant-implies-measure-invarinat} plays an important role in the proof of the following result, which gives a partial answer to the second part of  Question~\ref{Ques:inv-norms}.
\begin{theorem}[Ostrover and Wagner] \label{OW-theorem}
Let $ \| \cdot \|$ be a ${\rm Ham}(M,\omega)$-invariant norm on ${\cal A}$
such that $\| \cdot \| \leq C\| \cdot \|_{\infty}$ for some constant
$C$, but the two norms are not equivalent.\footnote{Two norms are said to be equivalent
 if  ${\frac 1 C} \, \| \cdot \|_1 \leqslant \| \cdot \|_2 \leqslant C \| \cdot \|_1$ for some constant $C>0$.}
%
Then the associated
pseudo-distance $d$ on ${\rm  Ham}(M,\omega)$ vanishes identically.
\end{theorem}
Although Theorem~\ref{OW-theorem} gives a partial answer to the second part of  Question~\ref{Ques:inv-norms}, prima facie, there might be ${\rm  Ham}(M,\omega)$-invariant norms on ${\cal A}$  which  are either strictly bigger than the $L_{\infty}$-norm, or  incomparable to it.
In a joint work with L. Buhovsky~\cite{BO} we showed that under the natural continuity assumption mentioned in Remark~\ref{Rmk-about-continuity} above, this cannot happen. Hence, up to equivalence of metrics,  Hofer's metric is unique. More precisely, 
%
%
%

\begin{theorem}[Buhovsky and Ostrover]  \label{Main-thm-BO} Let $(M,\omega)$ be a closed symplectic manifold.
Any $C^{\infty}$-continuous {\rm Ham}$(M,\omega)$-invariant pseudo-norm $\| \cdot \|$ on
${\mathcal A}$
is dominated  from above by the $L_{\infty}$-norm i.e., $\| \cdot \|
\leq C \| \cdot \|_{\infty}$ for some constant $C$.
\end{theorem}

Combining Theorem~\ref{Main-thm-BO} and Theorem~\ref{OW-theorem} above, we obtain:

\begin{corollary} For a closed symplectic manifold $(M,\omega)$, any bi-invariant Finsler pseudo-metric on ${\rm Ham}(M,\omega)$,
obtained by a pseudo-norm $\| \cdot \|$ on ${\mathcal A}$ that is
continuous in the $C^{\infty}$-topology, is either identically zero, 
or equivalent 
to Hofer's metric. In particular, any
non-degenerate bi-invariant Finsler metric on ${\rm Ham}(M,\omega)$ which
is generated by a norm that is continuous in the
$C^{\infty}$-topology  gives rise to the same topology on
${\rm Ham}(M,\omega)$ as the one induced by Hofer's metric.
\end{corollary}


In the rest of this section we briefly describe the strategy of the proof of
Theorem~\ref{Main-thm-BO} in the two-dimensional case. For the proof of the general case see~\cite{BO}. 
We start with two straightforward reduction steps. 
First, for technical reasons, 
%
%
%
we shall consider 
pseudo-norms on the space
$C^{\infty}(M)$,  instead of the space ${\mathcal A}$.
The original claim will follow, 
since any
Ham$(M,\omega)$ invariant pseudo-norm $\| \cdot \|$ on ${\mathcal
A}$ can be naturally extended to an invariant pseudo-norm $\| \cdot
\|'$ on $C^{\infty}(M)$ by 
$$\| f \|' := \| f- M_f \|, \ {\rm where \ } M_f = {\textstyle {\frac 1 {\rm Vol(M)}} \int_M f \omega^n}.$$ Note that if
$\| \cdot \|$ is continuous in the $C^{\infty}$-topology, then so is
$\| \cdot \|'$, and that the two norms coincide 
on the space ${\mathcal A}$. Second, by using a standard
partition of unity argument, we can reduce the proof of Theorem~\ref{Main-thm-BO}
to a ``local result", i.e., 
it is sufficient to prove the
theorem for Ham$_c(W,\omega)$-invariant pseudo-norms
on the space of compactly supported smooth functions $C_c^{\infty}(W)$, where $W=(-L,L)^{2}$ is an open 
square in ${\mathbb R}^{2}$ (see~\cite{BO} for the details).

\smallskip

The next step, which is one of the key ideas of the proof, is to define the ``largest possible"  Ham$_c(W,\omega)$-invariant norm on the space of compactly supported smooth functions $C_c^{\infty}(W)$. To this end, 
we fix a (non-empty) finite collection of functions ${\mathcal F} \subset C_c^{\infty}(W)$, and define: 
%
%
\begin{multline*}
 {\cal L}_{\mathcal F} := \Bigl \{ \sum_{i,k} c_{i,k} \, \Phi_{i,k}^* {f}_i \ | \ c_{i,k} \in {\mathbb R},
\ \Phi_{i,k} \in {\rm Ham}_c(W,\omega),  \\  \ {f}_i \in {\mathcal F},\  
{\rm and} \ \# \{(i,k) \, | \, c_{i,k} \neq 0 \} < \infty \Bigr  \}.
\end{multline*}
%
%
%
%
We equip  the space  ${\cal L}_{\mathcal F} $ with the norm $$ \| f \|_{{\cal L}_{{\mathcal F}}} = \inf
\sum |c_{i,k}|,$$ where the infimum is taken over all the
representations $f = \sum c_{i,k} \, \Phi_{i,k}^* {f}_i$ as above.
\begin{definition}
For any compactly supported function $ f \in C_c^{\infty}(W) $, let
\begin{equation*} 
 \| f \|_{{\cal F}, \, {\rm max}} = \inf \big\{ \liminf_{i \rightarrow \infty} \| f_i
\|_{{\mathcal L}_{{\mathcal F}}} \big\} ,\end{equation*} where the infimum is
taken over all subsequences $\{f_i\}$ in $ {\cal L}_{\mathcal F} $ which
converge to $f$ in the $C^{\infty}$-topology. As usual, the infimum
of the empty set is set to be $+ \infty$.
\end{definition}

The main feature of the norm $\| \cdot \|_{{\cal F}, \, {\rm max}}$ is 
that it dominates from above any other Ham$_c(W,\omega)$-invariant
pseudo-norm that is continuous in the
$C^{\infty}$-topology.  
\begin{lemma} \label{lemma-about-max-norm}
Let ${\mathcal F} \subset C_c^{\infty}(W)$ be a non-empty finite
collection of smooth compactly supported functions in $W$. Then  any
{\rm Ham}$_c(W,\omega)$-invariant pseudo-norm $ \| \cdot \| $ on $
C_c^\infty(W) $ that is continuous in the $ C^\infty $-topology
satisfies $$ \| \cdot \| \leqslant C \| \cdot \|_{{\cal F}, \, {\rm max}},$$
for some absolute constant $C$.
\end{lemma}

\begin{proof}[\bf Proof of Lemma~\ref{lemma-about-max-norm}]
Since the collection ${\mathcal F}$ is finite, set $C = \max \{ \| g \|  ; \, g \in {\mathcal F} \}$. For any $f =
\sum c_{i,k} \, \Phi_{i,k}^*  f_i \in {\cal L}_{\mathcal F}$, one
has
\begin{equation} \label{simpel-estimate1}  \|f \| \leq   \sum |c_{i,k}| \| \Phi_{i,k}^* f_i \|  \leq  C \sum |c_{i,k}|. 
\end{equation}
By the definition of $\| \cdot \|_{{\mathcal L}_{{\mathcal F}}}$, this immediately implies that $\|f\| \leq  C  \| f \|_{{\mathcal L}_{{\mathcal F}}}$.
The lemma now follows by combining~$(\ref{simpel-estimate1})$, the
definition of  $ \| \cdot \|_{{\cal F}, \, {\rm max}} $, and the fact that
the pseudo-norm $ \| \cdot \| $ is assumed to be continuous in the $
C^\infty $-topology.
\end{proof}

\smallskip

The next step, which is the
main part of the proof, is to show that for a suitable collection of
functions ${\mathcal F} \subset C_c^{\infty}(W)$, the norm $\| \cdot
\|_{{\mathcal F}, \, {\rm max}}$ is in turn bounded from above by the
$L_{\infty}$-norm. 
In light of the above, this would complete the proof of Theorem~\ref{Main-thm-BO} in the two-dimensional case. 
%
%
%
%
%

\smallskip

There are two independent components  in the
proof of this claim. First, we show that one can
decompose any $f \in C_c^{\infty}(W^2)$ with $\| f \|_{{\infty}} \leqslant 1$
into a finite combination $f = \sum_{i=1}^{N_0} \epsilon_j \Psi^*_j g_j$. Here, $ \epsilon_j \in \{ -1, 1 \} $, $\Psi_j \in {\rm Ham}_c(W^2,\omega)$, and $g_j$ are smooth rotation-invariant 
functions  whose $L_{\infty}$-norm is bounded by an absolute
constant, and
%
which satisfy certain other technical conditions (see
Proposition 3.5 in~\cite{BO} for the precise
statement). In what follows we call such functions ``simple
functions". We emphasize that $N_0$ is a constant independent of
$f$. Thus, we can restrict ourselves to the case where $f$ is a
``simple function''. In the second part of the proof, we construct
an explicit collection ${\mathcal F} = \{ {\mathfrak f_0},
{\mathfrak f_1} , {\mathfrak f_2} \} $, where ${\mathfrak f_i} \in
C_c^{\infty}(W^2),$ and $i=0,1,2$. 
Using an averaging procedure
(see the proof of Theorem 3.4 in~\cite{BO}), one can show that every
``simple function'' $f \in C_c^{\infty}(W^2)$ can be approximated
arbitrarily well in the $C^{\infty}$-topology by a sum of the form
$$ \sum_{i,k} \alpha_{i,k} \widetilde \Psi_{i,k}^* {\mathfrak f}_{k}, \ {\rm where \ }  \widetilde \Psi_{i,k} \in {\rm Ham}_c(W^2,\omega), \ k \in \{0,1,2 \}, $$ 
and such that $\sum | \alpha_{i,k} | \leq C \| f \|_{\infty}$ for
some absolute constant $C$. Combining this with the above definition
of $\| \cdot \|_{{\mathcal F}, \, {\rm max}}$, we conclude that
$ \| f \|_{{\mathcal F}, \, {\rm max}} \leq C \| f \|_{\infty}$ for every
$f \in C_c^{\infty}(W^2)$. Together with Lemma~\ref{lemma-about-max-norm}, this completes the proof of
Theorem~\ref{Main-thm-BO} in the 2-dimensional case.

\section{Some Open Questions and Speculations} \label{SEC:OQ}

\smallskip

\noindent{\bf Do symplectic capacities coincide on  the class of convex domains?} 
%
As mentioned above, since the time of Gromov's original work, a variety of symplectic capacities have been constructed and the
relations between them often lead to the discovery of surprising connections between symplectic geometry and Hamiltonian dynamics.
In the two-dimensional case, Siburg~\cite{Sib} showed that any symplectic capacity of  a compact connected domain with smooth boundary $\Omega \subset {\mathbb R}^2$ equals its Lebesgue measure. 
In higher dimensions symplectic capacities do not coincide in general. 
A 
theorem  by Hermann~\cite{Her} states that for any $n \geq 2$ there
is a bounded star-shaped 
domain $S \subset {\mathbb R}^{2n}$  with
cylindrical capacity $\overline c(S) \geq 1$, and arbitrarily small Gromov
radius $\underline c(S)$.  Still, for large classes of sets in
${\mathbb R}^{2n}$, including ellipsoids, polydiscs and convex Reinhardt
domains, all symplectic capacities coincide~\cite{Her}. 
In~\cite{V} Viterbo showed that for any bounded convex subset $\Sigma$
 of ${\mathbb R}^{2n}$ one has  $\overline c(\Sigma) \leq 4n^2 \underline c(\Sigma)$. 
Moreover, one has (see~\cite{Her,H2,V}) the following: 
\begin{conjecture} \label{conj-all-cap-coincide}
For any convex domain $\Sigma$ in
${\mathbb R}^{2n}$ one has $\underline c(\Sigma) = \overline c(\Sigma)$.
\end{conjecture}
This conjecture is particularly challenging due to the scarcity of examples of convex domains in which capacities have been  computed. 
Moreover, note that Conjecture~\ref{conj-all-cap-coincide} is stronger than Viterbo's conjecture (Conjecture~\ref{iso-per-conj} above), as the latter holds trivially for
the Gromov radius.

\smallskip

A somewhat more modest question in this direction is whether Conjecture~\ref{conj-all-cap-coincide} holds asymptotically, i.e., whether there is an absolute constant $A$ such that for any convex domain $K \subset
{\mathbb R}^{2n}$ one has $\overline c(K) \leq A \, \underline c(K)$.  It would be interesting to explore whether methods from asymptotic geometric analysis can be used to answer this question. 
\bigskip

\noindent{\bf Are Hanner polytopes in fact symplectic balls in disguise?}
Recall that Mahler's conjecture states that the minimum possible Mahler volume is attained by a hypercube. 
It is interesting to note that the corresponding product configuration, when looked at through symplectic glasses,  is in fact a Euclidean ball in disguise. More precisely, it was proved in \S 7.9 of~\cite{Schle} (cf. Corollary 4.2 in~\cite{LMS})  that the interior of the product of a hypercube $Q \subset {\mathbb R}^n_q$ and its dual body, the cross-polytope $Q^{\circ} \subset {\mathbb R}^n_p$, is symplectomorphic to the interior of a Euclidean ball $B^{2n}(r)  \subset {\mathbb R}^n_q \times  {\mathbb R}^n_p$ with the same volume.
On the other hand, as mentioned in Section~\ref{SEC:MAHLER} above, if Mahler's conjecture holds, then there are other minimizers for the Mahler volume  aside of the hypercube. 
For example, consider the class of Hanner polytopes. A $d$-dimensional centrally symmetric polytope $P$ is
a Hanner polytope if either $P$ is one-dimensional (i.e., a symmetric interval), or $P$ 
is the free sum or direct product of two (lower dimensional) Hanner
polytopes $P_1$ and $P_2$.
Recall that the free sum of two polytopes, $P_1 \subset {\mathbb R}^n$, $P_2 \subset {\mathbb R}^m$ is  a $n+m$ polytope defined by $P_1 \oplus P_2 = {\rm Conv} ( \{ P_1 \times \{0\} \}  \cup \{  \{0\} \times P_2 \}) \subset {\mathbb R}^{n+m}$. 
It is not hard to check (see e.g.~\cite{SR}) that  the volume product of the cube is the same as that of Hanner polytopes. Thus every Hanner polytope is also a candidate for a minimizer of the volume
product among symmetric convex bodies. In light of the above mentioned result from~\cite{Schle}, a natural question is the following: 
%
%
%
%
\begin{question} Is every Hanner polytope a symplectic image of a Euclidean ball?
\end{question}
More precisely, is the interior of every Hanner polytope  symplectomorphic to the interior of a Euclidean ball with the same volume?
A negative answer to this question would give a counterexample to Conjecture~\ref{conj-all-cap-coincide} above, since it would show that the Gromov radius must be different from the Ekeland-Hofer-Zehnder capacity. 

%

%

\bigskip

\noindent{\bf Symplectic embeddings of Lagrangian products.}
Since Gromov's work~\cite{Gr}, questions about symplectic embeddings have lain at the heart
of symplectic geometry (see e.g.,~\cite{B1,B3,Gu,Hu2, LMS, Mcd2,McPol,McSch, Schl1,Schle}). 
  These questions are usually notoriously difficult, and up to date most results concern  only the embeddings of balls, ellipsoids and polydiscs. 
Note that even for this simple class of examples, only 
recently  has it  become possible to specify exactly when a four-dimensional ellipsoid is embeddable in a ball (McDuff and Schlenk~\cite{McSch}), or in another four-dimensional ellipsoid (McDuff~\cite{Mcd2}).
For some other related works we refer the reader to~\cite{BH,CCFHR,CGK,FM, HK,HL,Op1}.
\smallskip

Since symplectic capacities can naturally be used to detect symplectic embedding obstructions, and in light of the results mentioned in Section~\ref{SEC:MAHLER} (in particular, Theorem~\ref{Main-Theorem-From-AKO}), 
it is only natural to try to extend the above list of currently-known examples, and study symplectic embeddings of convex ``Lagrangian products" in 
the classical phase space. The main advantage of this class of bodies is that the action spectrum can be computed via billiard dynamics. This 
property would presumably make it easier to compute or estimate the Ekeland-Hofer capacities~\cite{EH}, 
or Hutchings' embedded contact homology capacities~\cite{Hu1,Hu2}, in this setting. 
A natural first step in this direction would be to consider the embedding of the Lagrangian product of two balls into a Euclidean ball.
More precisely, 
%
%
%
%
%
to study the  function $\sigma: {\mathbb N} \rightarrow {\mathbb R}$ defined by 
$$ \sigma(n) = \inf  \bigl \{a \ |  B^n_q(1)  \times B^n_p(1)  \stackrel{\rm symp} \hookrightarrow B^{2n}(a)  \bigr \}.$$
To the best of the author's knowledge, the value of $\sigma(n)$ is unknown already for the case  $n=2$.

\bigskip

\noindent{\bf Acknowledgement:}
 I am  deeply indebted to Leonid Polterovich for generously sharing his insights and perspective on topics related to this paper, as well as for many inspiring conversations throughout the years. I have also benefited significantly from an ongoing collaboration with Shiri Artstein-Avidan, I am grateful to her for many stimulating and enjoyable hours working together. I would also like to thank Felix Schlenk and Leonid Polterovich for their valuable comments on an earlier draft of this paper.
%
%

%

%


\end{document}